\documentclass{article}
\usepackage[utf8]{inputenc}
\usepackage[T1]{fontenc}
\usepackage[margin=1.2in]{geometry}
\usepackage{lmodern}
\usepackage{subcaption}
\usepackage{mathtools}
\usepackage{amsthm,url}
\usepackage{todonotes}
\usepackage{mathrsfs,amssymb}  
\usepackage{enumitem}

\usepackage[capitalise]{cleveref}
\crefname{subsection}{subsection}{subsections}
\makeatletter

\renewcommand\section{\@startsection {section}{1}{\z@}
{-30pt \@plus -1ex \@minus -.2ex}
{2.3ex \@plus.2ex}
{\normalfont\normalsize\bfseries\boldmath}}

\renewcommand\subsection{\@startsection{subsection}{2}{\z@}
{-3.25ex\@plus -1ex \@minus -.2ex}
{1.5ex \@plus .2ex}
{\normalfont\normalsize\bfseries\boldmath}}

\renewcommand{\@seccntformat}[1]{\csname the#1\endcsname. }

\newtheorem{theorem}{Theorem}

\newtheorem{claim}[theorem]{Claim}

\newtheorem{lemma}[theorem]{Lemma}

\newtheorem{conj}[theorem]{Conjecture}

\newtheorem{question}[theorem]{Question}
\theoremstyle{definition}
\newtheorem*{definition*}{Definition}

\usepackage{pgfmath, pgfplots}
\pgfplotsset{width=10cm,compat=1.9}
\newcommand*{\myproofname}{Proof}
\newenvironment{claimproof}[1][\myproofname]{\begin{proof}[#1]}{\end{proof}}

\def\a{\alpha}
\def\eps{\varepsilon}
\def\F{\mathcal{F}}
 \newcommand{\Exp}{\,\mathbb{E}}
 \renewcommand{\Pr}{\,\mathbb{P}}
\newcommand*{\abs}[1]{\lvert #1\rvert}

\title{Better bounds for the union-closed sets conjecture using the entropy approach}
\author{Stijn Cambie\thanks{Extremal Combinatorics and Probability Group (ECOPRO), Institute for Basic Science (IBS), Daejeon, South Korea, supported by the Institute for Basic Science (IBS-R029-C4), E-mail: {\tt stijn.cambie@hotmail.com}} }
\begin{document}

\maketitle

\begin{abstract}
We improve the best known constant $\frac{3-\sqrt 5}{2}$ for which the union-closed conjecture is known to be true, by using dependent samples as suggested by Sawin and the entropy approach on this problem initiated by Gilmer. 
Meanwhile, we focus on the intuition behind this entropy approach and its boundaries. 
\end{abstract}

\section{Introduction}

The union-closed conjecture is a challenging conjecture in extremal set theory, see e.g.~\cite[sec.~32]{FT18}, which became famous due to its elegance.
A union-closed family $\F$ is a collection of sets such that the union of any two sets belongs to $\F$ as well.
The union-closed conjecture states that if $\F$ contains at least one nonempty set, then there is an element that belongs to at least half of the sets in $\F.$
This can be formally stated as follows, where we recommend the reader who is not familiar with some of the terminology, notations or definitions to first have a look at Subsection~\ref{subsec:not&def}.

\begin{conj}[Union-closed conjecture]\label{conj:UC}
   If $\F \not= \{\emptyset\}$ is a union-closed family with ground set $[n],$ then there exists an element $i \in [n]$ such that at least half of the sets in $\F$ contain $i$, i.e., $\abs{\F(i)} \ge \frac{\abs \F}{2}.$
\end{conj}

This would be tight by taking all subsets of a fixed ground set. Indeed, if $\F=2^{[m]},$ then every integer $i \in [m]$ appears in exactly $2^{m-1}$ sets of $\F.$ Possibly, these are essentially the only tight examples (duplicating elements does not change the structure), as also suggested in the blog on the Polymath project; \url{https://gowers.wordpress.com/2016/01/21/}. 

According to~\cite{BB14, BBE13, FT18}, the union-closed conjecture was already a folklore conjecture since the late 1960s or beginning 1970s, and was made well-known by Frankl, who rediscovered it in the late 1970s (1979 according to~\cite{F95}), and Ron Graham. Nonetheless, the first formal publication containing it might be~\cite{rival85}.
In the $90$s, it was proven in~\cite{Knill94,Wojcik99} that there is an element that appears in at least a $\Omega \left( \frac{ \log \abs \F }{\abs \F}\right)$ fraction of all sets.
The latter result is also implied by the union-closed size problem by Reimer~\cite{Reimer03}, which was fully resolved in~\cite{BBE13}.
In contrast to the union-closed size problem, which was solved in $10$ years, the union-closed conjecture is still open.
It has been proven in various specific cases, e.g. for certain random generated union-closed families it has been proven to be true with high probability~\cite{BB14}.
More on the history till $2015$, with other equivalent formulations of the conjecture, can be found in the survey~\cite{BS15}.
One equivalent formulation by~\cite{BCST15} states that every bipartite graph with at least one edge has at least one vertex in each bipartition class that belongs to at most half of the maximal independent sets.
At the beginning of $2017$, Karpas~\cite{Karpas17} proved the union-closed conjecture for families that contain roughly at least half of all sets.

Very recently, Gilmer~\cite{Gilmer22} proved the first linear bound using an elegant entropy-based method. 
As such he resolved the $\eps$-Union-Closed Sets Conjecture, as stated in~\cite{Hu17}.
Gilmer claimed that a tight version of his method could prove a fraction equal to $\frac{3-\sqrt 5}{2},$ which was soon verified by~\cite{AHS22,CL22,Sawin22}.
The tight version heavily depends on determining the minimum of a function $\frac{h(x^2)}{x h(x)}$, where $h$ is the binary entropy function.
Chase and Lovett~\cite{CL22} gave a clear, short proof using the minimum determined by~\cite{AHS22}. 
%The minimum is attained in $x=\frac{\sqrt 5-1}{2}$, which is the positive solution for $ x^2=1-x.$

%\noindent A sketch for a fourth variant of the proof with an additional claim, that can be verified as well, was given in~\cite{Pebody22}. 
A question and conjecture of Gilmer were soon answered in the negative, by~\cite{Sawin22, Ellis22}.
By working with approximate union-closed families in~\cite{CL22}, $\frac{3- \sqrt 5}{2}$ seemed possibly the best constant one could aim for with the idea of Gilmer~\cite{Gilmer22}.
Nevertheless, as suggested by Sawin~\cite{Sawin22}, this is not the case.
For this, we address the following question, from which the improved constant can be concluded later.
Note the inequality cannot be strict, since for any $\{0,1\}$-valued random variable equality does hold. 
\begin{question}[Sawin]\label{ques:Sawin}
    What is the maximum value $c$ for which there exists an $\a \in [0,1]$ such that the following is true?
    For every $p,q,r$ identically distributed $[0,1]$-valued random variables with expectation less than $c$, where $p$ and $q$ are independent, but $p$ and $r$ are not necessarily independent, we have 
    \begin{equation}\label{eq:Sawin}
        (1-\a)\Exp[ H(p+q-pq)] + \a\Exp\left[H\left(max\left(p,r,\min\left(p+r,1/2\right)\right)\right)\right] \ge \Exp[H(p)].
    \end{equation}
\end{question}
Yu~\cite{Yu22} considered the approach and question of Sawin in larger generality and derived bounds expressed in general optimisation forms. 

Our contribution consists in solving Question~\ref{ques:Sawin} exactly and as such improving the constant for which the union-closed conjecture is true.
The core content is written in Section~\ref{sec:possiblebounds}. Here, we start with explaining the entropy approach in general.
After that, we give some intuition why the bound can be improved despite the sharpness of the approximate union-closed conjecture by~\cite{CL22} and why the direct use of Sawin's idea cannot improve the constant too much.
For this, we provide the upper bound for $c$ in Question~\ref{ques:Sawin}, which later will turn out to be sharp.
As a last part of this section, in Subsection~\ref{subsec:sumproof}, we summarise the additional steps of the proof for the improved constant.

In the next section, Section~\ref{sec:betterconstant}, we prove the best bound of $c$ in Question~\ref{ques:Sawin}.
In Subsections~\ref{subsec:reduction} and~\ref{subsec:redSup_to4}, we reduce the possible probability distributions one has to consider and prove that the critical probability distributions have a support containing at most $3$ values.
This is as such the technical core to work out Sawin's idea and to answer Question~\ref{ques:Sawin}.
After this, we can express the problem as a minimisation problem of a function in $4$ variables and the optimisation problem can be verified with a computer-verification. 
In contrast to the detailed work in e.g.~\cite{AHS22} for the constant $\frac{3-\sqrt 5}{2}$, this is done slightly less rigorous. 
We need to take into account that the minimisation problem finds a local minimum and there is a finite computer precision involved.
By considering some plots, we note that there are two regimes to verify.
By observing that the extremal probability distribution is atomic (support has $2$ elements) in one regime,
we prove it more precisely for that regime and conclude.
In~\cref{subsec:ver_combinedideas}, we give a precise confirmation by combining our ideas with the strategy of Yu~\cite{Yu22}. As such, we can reduce the verification of~ \cref{ques:Sawin} to a minimisation problem in two variables, which can be solved numerically with graphical confirmation. The graphical confirmation gives information on the behaviour on local and global minima, which in principle is not the case in our previous strategy (3 variables), Yu~\cite{Yu22} (4 variables) and Liu~\cite{Jingbo23} (9 variables for a slightly further improved constant).
Since the constant is not $\frac 12$ and another core method will be needed for the full resolution, an even more rigorous analysis than has been done 
is unnecessary, as it does not contribute to further understanding the underlying principles.
Finally, in Section~\ref{sec:proof} we summarise the proof for the improved constant on the union-closed conjecture, based on the answer for Question~\ref{ques:Sawin}. That is, for $c\sim 0.3823455$ we prove the following theorem.
\begin{theorem}\label{thr:improvement}
    Let $\F \subset 2^{[n]}$ be a nonempty union-closed family.
    Then there is some element $i \in [n]$ that appears in at least $c\abs{\F}$ many sets of $\F.$
\end{theorem}

So to recap, in this paper, we prove that by taking a linear combination of probability distributions, not all of them being independent, the bound on the union-closed conjecture derived with the entropy approach of Gilmer~\cite{Gilmer22} can be improved slightly. 

\subsection{Terminology and notation}\label{subsec:not&def}

In this subsection, we collect some basic notation and definitions used in extremal set theory and the entropy method used in the papers related to the work of Gilmer.
For more of this, we refer to~\cite{FT18} for extremal set theory and~\cite{CT06} for entropy.

The standard finite set of cardinality $n$ is denoted with $[n]=\{1,2,\ldots, n\}.$
A subset of $[n]$ is a set containing some elements from $[n]$ and can possibly be the empty set $\emptyset.$
A collection of subsets of $[n]$ is called a set system or family $\F$.
The family $2^{[n]}$ (the power set of $[n]$) contains all $2^n$ possible subsets of $[n]$ and in general a family $\F \subset 2^{[n]}$  is a subset of this power set.
A uniform family is a family whose sets all have the same cardinality $k.$
The largest $k$-uniform family on the ground set $[n]$ is $\binom{[n]}{k}=\{ A \subset [n] \colon \abs{A}=k\}.$
Similarly we use $\binom{[n]}{\ge k}$ to denote $\{ A \subset [n] \colon \abs{A}\ge k\}.$

A family $\F$ is called \textbf{union-closed} if for every $A, B \in \F$ also the union $A \cup B \in \F.$
Using $\F \cup \F=\{ A \cup B \colon A,B \in \F\},$ a family $\F$ is union-closed if and only if $\F \cup \F=\F.$
It is approximate union-closed if the latter is true for almost every choice of $A, B \in \F.$
The sets containing a fixed element $i$ are essentially presented by the family $\F(i)=\{A \backslash \{i\} \colon i \in A \in \F\}.$
Similarly, $\F(\overline i)=\{A \colon i \not \in A \in \F\}.$

We will use some Landau-notation.
Given two functions $f,g \colon \mathbb R \to \mathbb R$, we write
\begin{itemize}
    \item $f=o(g)$ if $\lim_{x \to \infty} f(x)/g(x)=0,$
    \item $f \sim g$ if $\lim_{x \to \infty} f(x)/g(x)=1.,$
    \item $f=O(g)$ if there is a constant $C$ such that $\lim_{x \to \infty} \abs{f(x)/g(x)}<C,$
    \item $f=\Omega(g)$ if $g=O(f),$
    \item $f=\Theta(g)$ if $g=O(f)$ and $f=O(g).$
\end{itemize}

A random variable $X$ can be related with the probability distribution of the outcomes.
The entropy $H(X)$ of a discrete random variable $X$ equals the Shannon entropy of its probability sequence and is denoted by $H(X).$
The support of a probability distribution or random variable, is the subset of the possible outcomes which have positive probability (density function).
If the support of $X$ is a finite set $A$, and each outcome $x \in A$ has a probability $p_x,$ then 
$$H(X)=-\sum_{x \in A} p_x \log_2 p_x.$$
It is a fundamental result, a corollary of Jensen's inequality, that this is bounded by 
$\log_2 \abs{A}.$
For a random variable with only two outcomes, occurring with probabilities $p$ and $1-p$, we denote the entropy with the binary entropy function $h(p)=-( p\log_2 p + (1-p) \log_2(1-p)).$\footnote{For clarity, the binary entropy function is denoted with $h$ and the entropy function of a variable with $H$.}
%We note that all entropies used can be considered as binary versions, except from the conditional versions.
A conditional entropy of a random variable $Y$ given $X$ can be computed as
$$H(Y|X)\ =-\sum _{x\in X,y\in Y} \Pr(x,y)\log_2 {\frac {\Pr(x,y)}{\Pr(x)}}.$$
Here $\Pr(x,y)$ is the probability on the outcome $\{X=x, Y=y\}.$
The expectation of a random variable $X$ is $\Exp[X]= \sum_{x \in A} p_x x.$
We use $\vee$ for the logical or, that is $x\vee y$ implies that $x$ or $y$ has to be satisfied.

\section{Preliminaries}\label{sec:possiblebounds}

In this section, we begin by presenting the main idea behind the entropy approach introduced by Gilmer~\cite{Gilmer22}. We then provide some intuition to explain why the examples from~\cite{Gilmer22, CL22}, which initially suggest that the constant $\frac{3-\sqrt{5}}{2}$ cannot be improved using this method, can actually be refined. Additionally, we (try to) demonstrate why the linear combination in Question~\ref{ques:Sawin} is somewhat necessary\footnote{i.e., variants will not be simpler} and argue that the constant cannot be significantly improved by solving this question\footnote{E.g.~\cite{Jingbo23} gave a variant with tiny improvement}. Finally, we summarise the key insights behind the proof. More examples and explanations can also be found in the survey~\cite{Cambie23}.

\subsection{Why the entropy approach works for the union-closed conjecture}

The contraposition of the union-closed conjecture states that if every element $i \in [n]$ appears in strictly less than half of the sets of a family $\F \subset 2^{[n]}$, then $\F$ is not union-closed and thus $\abs{\F} < \abs{\F \cup \F}.$
The idea behind the entropy approach initiated by Gilmer~\cite{Gilmer22} is that if one can find a way to sample sets from $\F \cup \F$ such that the entropy is strictly larger than $\log_2 \abs{\F},$ one can conclude that $\abs{\F} < \abs{\F \cup \F}.$ The latter since the entropy of a random variable with $N$ possible outcomes is bounded by $\log_2 N.$
Proving that $H(A \cup B) >\log_2 \abs \F$ where $A,B$ are sampled from $\F$, is hard when lacking information on $\F.$ It is easier to compare $H(A \cup B)$ with $H(A)$. For the conclusion to hold, $A$ needs to be sampled uniform random from $\F$.
Gilmer~\cite{Gilmer22} did this by taking two uniform independent (iid) random samples $A,B$ from $\F$ and considering the entropy of the union $A \cup B.$ 

For alternatives, one could sample $B$ non-uniform from $\F$, or have two sampled random variables which are not independent. In that case one tries to do this in such a way that $A \cup B$ is as uniformly distributed over $\F \cup \F$ 
as possible.
Exactly uniformly distributed seems impossible.
If $\F$ would be union-closed and $A$ is uniform random, $A \cup B$ needs to be equal to $A$ to ensure that $A \cup B$ is uniformly distributed as well.
% , but for a non-union-closed family it should give all of $\F \cup \F$ as possible outcomes.
As a concrete example, when $\F= \binom{[2]}{\ge 1},$ then $\Pr(A=\{1\}) \Pr(B=\{1\})=\frac 13 = \Pr(A=\{2\}) \Pr(B=\{2\})$ is impossible.
Once one is sampling dependent samples, one would need additional ideas to know more about the conditional probability distributions.
Sawin~\cite{Sawin22} gave the most natural choice when sampling the sets element-wise.

\subsection{Observations on approximate union-closed set system}

Let $\psi=\frac{3-\sqrt 5}{2}$ be the smallest root of $2x-x^2=1-x \Leftrightarrow x^2-3x+1=0$ and $g \colon \mathbb N \to \mathbb N \colon n \mapsto g(n)$ be a function which is both $o(n)$ and $\omega(n^{0.5}).$
In~\cite{CL22}, the authors take $g(n) \sim n^{2/3}.$
The family 
$$
\F=\{ A \subset [n] \colon \abs{A} =\psi n +g(n) \vee \abs{A} \ge (1-\psi) n\}\\
=\binom{[n]}{\psi n +g(n)} \cup \binom{[n]}{\ge (1-\psi) n}
$$ is an approximate union-closed family.
That is, the union of two sets (iid, independent identically distributed, uniform random chosen) in $\F$ belongs with high probability to $\F$ as well.
Nevertheless, since $\psi > \frac{1}{3}$, one can also observe that $\F \cup \F = \binom{[n]}{\geq \psi n + g(n)}$. That is, only a small proportion of $\F \cup \F$ belongs to $\F$. As such, this construction, which might suggest that improving the constant $\psi$ is impossible with the entropy approach, does not necessarily imply this conclusion. If unions of two sets in $\F$ are taken with a non-uniform measure or in a dependent manner such that sets with smaller unions are more likely to be selected, the initial argument about the sharpness of this construction no longer holds. The entropy of the union of two random variables, when these variables take values in $\F$ in a dependent way, can indeed be larger. With the dependency proposed by Sawin, the union of two sets from the approximate union-closed family mentioned above will almost surely have a size of $(0.5 + o(1))n$, for example.

\subsection{Upper bound for Question~\ref{ques:Sawin} and limits on the approach of Sawin}\label{subsec:uppbound}

Let $h(x)=-x \log_2(x)-(1-x)\log_2(1-x).$
Let the roots of the function $h(x)(2-h(x))-h(2x-x^2)$ be $0<b_1<b_2<1.$
Then $b_1 \sim 0.139499451909862$ and $b_2 \sim 0.329454738503037 .$
Choose $b=b_2$ and $a=\frac{1-h(b)}{2-h(b)}\sim 0.0788772927059232.$

Let $p,q,r$ be three identically distributed $[0, 1]$-valued random variables, where 
$\Pr(p=1)=a$ and $\Pr(p=b)=1-a,$ such that $p$ and $q$ are independent and $p$ and $r$ are as negatively correlated as possible in the sense that 
$\Pr(p=r=1)=0.$
For these choices, we have $\Exp[p]=0.382345533366703$\footnote{The computation can be found in \url{https://github.com/StijnCambie/UCconjecture/blob/main/Sharpness.sagews}} and $(1-a)^2 h(2b-b^2)=(1-2a)=(1-a)h(b)$ which is equivalent with
$$\Exp[H(p+q-pq)]=\Exp\left[H\left(max\left(p,r,\min\left(p+r,1/2\right)\right)\right)\right] =\Exp[H(p)].$$
Hence no linear combination satisfies Equation~\ref{eq:Sawin} with a strict inequality and local perturbations (increasing $a$) will result in counterexamples when $c$ is allowed to be slightly larger.
%Furthermore, if we increase $a$ by a small number $\eps>0,$ then $$\Exp[H(p+q-pq)],\Exp\left[H(max(p,r,\min(p+r,1/2)))\right] <\Exp[H(p)].$$

On a different note, we observe that taking a linear combination, as done in Equation~\ref{eq:Sawin}, will be necessary to improve the constant $\frac{3-\sqrt{5}}{2}$ in the progress on the union-closed conjecture. To see this, observe that $\Exp\left[H(\max(p, r, \min(p+r, 1/2)))\right] < \Exp[H(p)]$ for $p$ and $r$ identically distributed with $\Pr(p = r = 1) = 0$, $\Pr(p = b) = 1 - a$, and $\Pr(p = 1) = a$, where $b = \frac{1}{4}$ and $a = \frac{1 - h(b)}{2 - h(b)} + \epsilon$ for some small $\epsilon > 0$. Since $\Exp[p] < 0.37 < \frac{3-\sqrt{5}}{2}$, considering the single term alone would not lead to an improvement.

From these observations, one can conclude that one cannot aim to prove the result with a constant better than $0.382345533366703$ with the exact suggested approach of Sawin.

\subsection{Summary of the proof}\label{subsec:sumproof}

The idea from Sawin~\cite{Sawin22} is to sample sets twice element-wise. Here iteratively, one samples $A \cap [k]$ and $B \cap [k]$ for $0 \le k \le n$ based on the probability that a set in $\F$, given the intersection with $[k]$, would contain the additional element $k+1$.
The sampling of the element $k+1$ for $A$ and $B$ can then be done in a dependent manner, ensuring that both $A$ and $B$ are uniform samples over $\F.$
Once some elements are sampled, control over the conditional probabilities (the distribution) is lost, so we assume the worst-case scenario. Since the worst-case scenarios for the two different strategies differ, a better bound is obtained by taking a linear combination of these two different ways of sampling in $\F \cup \F.$
As the entropy of the whole sample can be determined by summing (conditional) entropies for every element $k \in [n]$, it is sufficient to prove the inequality for these entropies for for a single element.
At this point, the problem reduces to a question that depends purely on the probability distribution of random variables and their expectations.

To attack that question, we perform local optimisation to find properties of an optimal distribution by redistributing the probability mass function and verifying convexity and concavity. As such, we reduce the problem to a simpler inequality involving only $4$ unknown parameters associated with a probability distribution whose support contains at most $3$ elements. 
Combining with the work of Yu~\cite{Yu22}, this is even further improved to two cases depending on two variables each.
The final minimisation problems are verified numerically with a computer program in two ways, together with a plot showing that we obtain the global minimum by the proposed atomic probability distribution.% (the support contains only two elements).

\section{Proof for the optimal constant of Question~\ref{ques:Sawin}}\label{sec:betterconstant}

In this section, we prove that the maximum constant $c$ for Question~\ref{ques:Sawin} is approximately $0.382345533366$ (the value derived in Subsection~\ref{subsec:uppbound}).
We do so by proving it for the optimal choice of $\a$, $\a \sim 0.0356069$.
The latter is obtained from comparing derivatives of the atomic solutions with support on $\{x,1\}$ and expectation $c$ around $x=b$.

That is, the probability $p$ is given by $\Pr(p=1)=\frac{c-x}{1-x}$ and $\Pr(p=x)=1-\frac{c-x}{1-x}$.
Remembering that $p,r$ are negatively correlated, we let 
\begin{align*}
    g_1(x)&=\Exp[ H(p+q-pq)] - \Exp [ H(p)]= \Pr(p=x)^2 h(2x-x^2)-  \Pr(p=x) h(x) \mbox{ and }\\    g_2(x)&=\Exp\left[H\left(max\left(p,r,\min\left(p+r,1/2\right)\right)\right)\right]- \Exp[H(p)]\\&= \left(1-2\Pr(p=1)\right) -  \Pr(p=x) h(x)
\end{align*}
Then~\cref{eq:Sawin} is equivalent with $(1-\a)g_1(x)+\a g_2(x) \ge 0$ and thus $\a$ need to satisfy $(1-\a)g'_1(b)+\a g'_2(b)=0$ or equivalently $\a= \frac{ g'_1(b)}{g'_1(b)-g'_2(b)}.$

We first show that it suffices to consider the case where the probability distributions $p, q, r$ are not supported on $(0.5,1).$
Next, by performing analytical computations on the behaviour of the functions in two intervals, similar to what Sawin~\cite{Sawin22} did, we reduce the problem to finding the minima of a continuous function in three variables on a bounded region, and subsequently to two variables.
As such, the remaining problem is a minimisation problem for which the statement can be exactly verified with the help of a computer.
% Since an exact rigorous calculus proof would not provide further insight into the core problem (Conjecture~\ref{conj:UC}) and the entropy approach, while we know exactly what is going on, we consider this as sufficient evidence for the exact bounds of the idea proposed by Sawin~\cite{Sawin22}.
% This implies that an exact rigorous calculus proof is missing. Since this would not provide further insight into the core problem (Conjecture~\ref{conj:UC}) and the entropy approach, and the improvement of the constant is small, 
% he verification by computer seems sufficient to us as evidence for the exact bounds of the idea proposed by Sawin~\cite{Sawin22}.

\subsection{Reduction of support of the probability distribution}\label{subsec:reduction}

First we prove that it is sufficient to consider $[0, 1]$-valued random variables which do not attain values in $(0.5,1).$ 

For readability and since it is sufficient to consider finite supported measures for the application on Conjecture~\ref{conj:UC}, we prove the following lemmas only in the case of discrete probability distributions.
The proof of the following lemma can be modified for general probability distributions by replacing sums by integrals and probability $\Pr$ by probability distribution $\mu$.
% and additional care for some subtleties that arise.

\begin{lemma}\label{lem:red1}
    Assume $p,q$ are independent identically distributed (iid) $[0,1]$-valued random variables with expectation $\Exp[p]=c \le 0.39$ such that $\Pr(p=y)>0$ for some $1/2<y<1$.
    Then the modified common probability distribution $p',q'$ of $p,q$ for which $\Pr(p'=1)=\Pr(p=1)+(2y-1)\Pr(p=y), \Pr(p'=y)=0, \Pr(p'=0.5)=\Pr(p=0.5)+(2-2y)\Pr(p=y)$ and $\Pr(p'=y')=\Pr(p=y')$ for every $y'\in [0,1]\backslash\{0.5,y,1\}$ satisfies
    \begin{itemize}
        \item $\Exp[p']=\Exp[p]$ and
        \item $w\Exp[H(p')]-\Exp[H(p'+q'-p'q')]>w\Exp[H(p)]-\Exp[H(p+q-pq)]$ for every $w \le 1.044$, where $p',q'$ are iid.
    \end{itemize}
\end{lemma}
\begin{proof}
The first part is immediate since the choice of redistribution is chosen in such a way that the following two linear combinations are true;
$(2y-1)+(2-2y)0.5=y$ and $(2y-1)+(2-2y)=1$. The latter to ensure that we still have a probability distribution.
Hence it remains to prove the second part.
For this, we first take a very small value $\eps=\frac{\Pr(p=y)}{N}$ by choosing a large positive integer $N$. Let $I$ be the support (set with all values $x$ for which $\Pr(p=x)>0$) with $1/2$ and $1$ included as well.

First, we do the redistribution of only an $\eps$-fraction in the probability distribution, that is
$\Pr(p'=1)=\Pr(p=1)+(2y-1)\eps, \Pr(p'=y)=\Pr(p=y)-\eps, \Pr(p'=0.5)=\Pr(p=0.5)+(2-2y)\eps.$
Now $\Exp[H(p'+q'-p'q')]-\Exp[H(p+q-pq)]$ is equal to $\eps \lambda + O(\eps^2)$, where $\lambda$ equals
\begin{align*}
    2 \sum_{x \in I} \left((2-2y) h(0.5(1-x)) - h((1-y)(1-x))\right)\Pr(p=x).
\end{align*}
Here we have used that $h(x+y-xy)=h((1-x)(1-y))$ (by symmetry of $h$ and $1-(x+y-yx)=(1-x)(1-y)$) and $h(0)=0$.
Let $g(x)=(2-2y) h(0.5(1-x)) - h((1-y)(1-x)).$
Note that 
\begin{align*}
\ln 2 \frac{d}{dx} g(x)&= (y-1)\ln \left( \frac{(1-y)(1+x)}{x+y-xy} \right)> 0 \mbox{ and}\\
\ln 2 \frac{d^2}{dx^2} g(x)&= -\frac{(1-y)(2y-1)}{(x+1)(x+y-xy)}<0
\end{align*}
since $0<(1-y)(1+x)<x+y-xy$ for $1 >y>0.5$ and every $1\ge x\ge 0.$
Due to Jensen's inequality for the concave function $g$, $\lambda=2\Exp[g(p)]$ is upper bounded by
$2g(c).$
This upper bound is independent of $\Pr(p=y).$
Hence we can do this $N$ times and conclude that for $p', q'$ distributed as in the lemma, we have
$\Exp[H(p'+q'-p'q')]-\Exp[H(p+q-pq)]\le 2g(c) \Pr(p=y)+O(\eps).$
It is also straightforward to compute that 
$\Exp[H(p)]-\Exp[H(p')]=\Pr(p=y)\left(h(y)-(2-2y)h(0.5)\right)=-g(0)\Pr(p=y).$
Finally, it suffices to prove that 
$$2g(c)-g(0)<0$$ since then $\eps$ can be chosen sufficiently small such that after adding the $O(\eps)$ term, it is still negative.
Since $g$ is an increasing function and $g(0)<0$ (due to $h$ being concave), it suffices to prove that $2g(0.39)-1.044g(0)<0.$
This is the case for every $\frac 12 <y < 1.$\footnote{Verification at \url{https://github.com/StijnCambie/UCconjecture/blob/main/reduction\_UC.sagews}}
\end{proof}

\begin{lemma}\label{lem:red_part2}
Let $p,r$ be identically distributed $[0,1]$-valued random variables, not necessarily independent.
Then one can modify the underlying common probability distribution by distributing the probability mass function on $(0.5,1)$ over $0.5$ and $1$ such that $\Exp[p]$ is the same and $$\Exp\left[H\left(max\left(p',r',\min\left(p'+r',1/2\right)\right)\right)\right] \le \Exp\left[H\left(max\left(p,r,\min\left(p+r,1/2\right)\right)\right)\right].$$
\end{lemma}
\begin{proof}
    We do the following procedure as long as there is some value in $(0,1)$ with positive probability.
    Let $y=\max\{y \mid 0.5<y<1 \wedge \Pr(p=y)>0\}$ be the largest value in $(0,1)$ with positive probability and $y'$ the second largest such value, or $0.5$ if no other value in $(0,1)$ has positive probability.
    We distribute the probability mass function of $y$ over $y'$ and $1$ (such that we still end with a probability measure). 
    We let $p'$ and $r'$ be dependent as before, with the corresponding distribution taken into account (made clear below).
    We claim that the considered quantity $\Exp\left[H\left(\max\left(p,r,\min\left(p+r,1/2\right)\right)\right)\right]$ did not increase by doing so.
    If $\Pr(p=y,r=y)>0$,
    we increase $\Pr(p=y',r=y')$ and $\Pr(p=1,r=1)$ accordingly and conclude by concavity of $h$.
    If $\Pr(p=y,r=1)>0$, then $\max\{p,r,\min(p+r,1/2)\}=1$ both before and after the local adaptation of $p$ (similarly when $p$ and $r$ are switched) and so there is no change by this term.
    If $\Pr(p=y,r=z)>0$, for some $z\le y'$, then we conclude again by concavity of $h.$
    By iterating this process, the probability measure on $(0.5,1)$ is distributed over $0.5$ and $1$ and the condition in the lemma is satisfied.
\end{proof}

Now, assume there are random variables $p,q,r$ satisfying the conditions of Question~\ref{ques:Sawin} for which $\Exp[ H(p)]\le c$ and 
$(1-\a)\Exp[ H(p+q-pq)] + \a\Exp\left[H\left(max\left(p,r,\min\left(p+r,1/2\right)\right)\right)\right] \le \Exp[H(p)]$ for some $\a \in [0,1].$
\noindent Next, we consider the modified random variables $p',q',r'$,
where the probability distribution is iteratively adapted by distributing the probability $\Pr(p=y)$ for some $0.5<y<1$ over $\Pr(p=0.5)$ and $\Pr(p=1)$.
Since $1.044>\frac{1}{1-\a},$ Lemma~\ref{lem:red1} implies that 
$\Exp[H(p')]-(1-\a)\Exp[H(p'+q'-p'q')]>\Exp[H(p)]-(1-\a)\Exp[H(p+q-pq)]$.
Also $\Exp\left[H\left(max\left(p',r',\min\left(p'+r',1/2\right)\right)\right)\right] \le \Exp\left[H\left(max\left(p,r,\min\left(p+r,1/2\right)\right)\right)\right]$
for the natural choice of the adapted dependency of $p$ and $r$ by Lemma~\ref{lem:red_part2}.
Thus, if there are probability distributions $p,q$ and $r$ for which Equation~\ref{eq:Sawin} is not satisfied for some value of $c$, then these distributions must have support disjoint from $(0.5,1)$.

\subsection{Reduction to small support of the probability distribution}\label{subsec:redSup_to4}

In the previous subsection, we established that for~\cref{ques:Sawin}, it is sufficient to consider probability distributions whose support does not include values in $(0.5,1)$, and we now further restrict the support to at most three elements.

First, we observe that the quantity $\Exp\left[H\left(max\left(p,r,\min\left(p+r,1/2\right)\right)\right)\right]$ is minimised (under the condition that $p$ and $r$ have the same fixed distribution) when $\Pr(p=r=1)=0$.
If $\Pr(p=r=1)$ and $\Pr(p=x, r=y)>\epsilon>0$ for some values $0 < \max{x, y} < 1$, we modify the probability distribution by increasing $\Pr(p=x, r=1)$ and $\Pr(p=1, r=y)$ by $\epsilon$, and decreasing $\Pr(p=r=1)$ and $\Pr(p=x, r=y)$ by $\epsilon$. This decreases the expectation of the entropy function $\Exp\left[H\left(max\left(p,r,\min\left(p+r,1/2\right)\right)\right)\right]$.
Note that in the remaining case $\Pr(p=r=0)$ would be the only other positive probability and so the whole expectation $\Exp\left[H\left(max\left(p,r,\min\left(p+r,1/2\right)\right)\right)\right]=\Pr(p=r=0) h(0)+ \Pr(p=r=1) h(1)$ is zero. 
Similarly, we can make analogous modifications, decreasing $\Pr(p=r=1)$ and $\Pr(p=0, r=0)$ with $\eps=\Pr(p=r=1)$ and increasing $\Pr(p=0, r=1)$ and $\Pr(p=1, r=0)$ with $\eps$.
The latter is possible since we assumed $\Pr(p=1) \le \frac 12.$
Thus, without loss of generality, we may assume that $\Pr(p=r=1)=0$.

Now, by the result of the previous subsection,~\cref{subsec:reduction}, whenever $p,r<1,$ we have $p,r \le \frac 12$ and hence $max\left(p,r,\min\left(p+r,1/2\right)\right)=\min\left(p+r,1/2\right).$

For the remainder of this subsection, let $\Pr(p=1)=a$.
Define $x_0$ as the $(1-2a)$-quantile of $p$, i.e., the smallest value satisfying $\Pr(p \le x_0) \ge 1-2a.$

\begin{lemma}\label{lem:quantile_prcomparison}
    We can assume that $\Pr(p=r)=1-2a$ and this happens exactly for the $(1-2a)$-quantile $x_0$, that is, for every $x<x_0,$ we have $\Pr(p=r=x)=\Pr(p=x)$
and $\Pr(p=r=x_0)=1-2a-\Pr(p<x_0).$
\end{lemma}
\begin{proof}
To prove the statement, we show that probability mass can be redistributed without increasing the studied expectation.

Since $h$ is an increasing function on $[0,1/2]$, we can assume that the values $x,y \in [0,1/2]$ for which $\Pr(p=x,r=y)>0$ satisfy $x,y \le x_0$. In particular, for $x_0 < x \leq 1/2$ and $y_0 < y \leq 1/2$, we have the following condition: if $\Pr(p=x, r=z)>0$ or $\Pr(p=z, r=y)>0$, then $z=1$. This cancels the largest values in $[0,1/2]$, as their combination with $1$ results in a contribution of $0$ due to $h(1)=0$.

If $\Pr(p=x,r=y), \Pr(p=x',r=y')\ge \eps>0,$ where $x<x'<x_0\le 1/2$ and $1/2 \ge x_0 \ge y>y'$, we can decrease $\Pr(p=x,r=y)$ and $ \Pr(p=x',r=y')$ with $\eps$ and increase $\Pr(p=x,r=y'),\Pr(p=x',r=y)$ with $\eps.$
The studied expectation does not increase, by the follow claim.

\begin{claim}
    The function $g\colon [0,1] \to [0,1] \colon x \mapsto h(min(x,1/2))$ is a concave function.
    For all $x,x',y,y' \in [0,1/2]$ such that $x<x'$ and $y>y'$, we have
    $g(x+y) + g(x'+y') \ge g(x+y') +g(x'+y) .$
\end{claim}
\begin{claimproof}
    The second derivative of $g$ equals that one of $h$ on $[0,1/2)$ and is therefore strictly negative on this interval.
    The second derivative of $g$ is zero for $ x \ge 1/2.$
    
    Since $(x+y)+(x'+y')=(x+y')+(x'+y)$ and $x+y'< \min\{x+y,x'+y'\}$ and $\max\{x+y,x'+y'\}< x'+y,$ the pair $\{x'+y,x+y'\}$ majorises $\{x+y,x'+y'\}.$ The claim now follows from Karamata's inequality~\cite{Kar32}.    
\end{claimproof}

We conclude that we can assume that $\Pr(p=r)=1-2a$ and this happens exactly for the $(1-2a)$-quantile, that is, for every $x<x_0,$ we have $\Pr(p=r=x)=\Pr(p=x)$
and $\Pr(p=r=x_0)=1-2a-\Pr(p<x_0).$
\end{proof}

Next, we use the same approach as Sawin.
% Let $\a \in (0,1)$ and $c<\frac 12$ be fixed constants. 
Let $\mu$ be a probability distribution which minimises 
\begin{equation}\label{eq:2}
        H_{\mu} = (1-\a)\Exp_{(p,q) \sim \mu \times \mu}[ H(p+q-pq)] + \a\Exp'_{p \sim \mu} \left[H\left(\min\left(2p,1/2\right)\right)\right] -\Exp_{p \sim \mu}[H(p)]
\end{equation}
among all probability distributions with expectation bounded by $c$; $\Exp_{p \sim \mu}[H(p)]\le c.$
Let $\Pr_{p\sim \mu}(p=1)=a$ and let the $(1-2a)$-quantile of $\mu$ be $x_0$.
Such a distribution exists, as explained in the proof of Sawin's Lemma 3.
Here $\Exp'$ has to be interpreted as the expectation over the $(1-2a)$-quantile (due to~\cref{lem:quantile_prcomparison}).
\begin{lemma}\label{lem:reduced_minimizingfunction}
    The probability distribution $\mu$ also minimises
    \begin{align*}
        2(1-\a) \Exp_{(p,q) \sim \mu \times \nu} [H(p+q-pq)] -\Exp_{p \sim \nu} H[p] +\a \Exp'_{p \sim \nu} \left[H\left(\min\left(2p,1/2\right)\right)\right]\\
        =\Exp_{q \sim \nu} \left(  2(1-\a) \Exp_{p \sim \mu } [H(p+q-pq)] - H(q) \right) + \Exp'_{q \sim \nu} \left[H\left(\min\left(2q,1/2\right)\right)\right]
    \end{align*} among all probability measures $\nu$ for which the $(1-2a)$-quantile is $x_0$, $\Exp_{p \sim \nu} H[p] \le c$ and $\Pr_{p\sim \nu}(p=1)=a$.
\end{lemma}
\begin{proof}
Consider the combination $\mu'=(1-\eps)\mu+\eps \nu$, which has the same values for $x_0$ and $a=\Pr(p=1)$.
By definition of $\mu$ being a minimiser, $H_{\mu'}-H_{\mu} \ge 0$.
Now $\frac{H_{\mu'}-H_{\mu}}{\eps}$ equals, up to a $O(\eps)$ function,
\begin{align*}
&2(1-\a)\left(  \left(\Exp_{(p,q) \sim \mu \times \nu}- \Exp_{ (p,q) \sim\mu \times \mu}\right)[H(p+q-pq)]\right)-(\Exp_{p \sim \nu}-\Exp_{p \sim \mu}) [H(p)]\\
&+\a \left( 
(\Exp'_{p \sim \nu}-\Exp'_{p \sim \mu}) \left[H\left(\min\left(2p,1/2\right)\right)\right] \right) .
\end{align*}
So by taking $\eps$ sufficiently small, we conclude.
\end{proof}

Now for every fixed constant $0 \le q \le 1,$ the function $F_{\mu}(q)= \Exp_{p \sim \mu}[2(1-\a) H(p+q-pq)- h(q)]$ satisfies $\frac{d}{dq} \left( q(1-q) \frac{d^2}{dq^2} F_{\mu}(q) \right)<0$, as verified in the proof of~\cite[Lem.~3]{Sawin22}.
% \footnote{Here $q$ is a single variable instead of random variable.}
By direct computation,  we verify that
$\ln 2 \frac{d^2}{dq^2} h\left(2q\right)=\frac {-2(1-q)}{(1-2q)}$
and $\frac{d}{dq} \left( q(1-q)\ln 2 \frac{d^2}{dq^2} h\left(2q\right) \right) =\frac{d}{dq} \left( 
\frac {-2}{(1-2q)q} \right)  =\frac{-2}{(1-2q)^2} <0. $
Hence $q(1-q)\ln 2 \frac{d^2}{dq^2} \left( F_{\mu}(q) +h\left(2q\right) \right)$ is a strictly decreasing function.
This implies that if we consider the function $F_{\mu}(q) +h\left(2q\right)$ on the interval $I_1=[0, min\{x_0,1/4\}]$ 
and $F_{\mu}(q)$ on the interval $I_2=[min\{x_0,1/4\},1/2]$ separately, we observe that the second derivative of each function behaves in one of three ways: it is either strictly positive, strictly negative, or changes sign at a critical point 
$z_1$ or $z_2$.
 
% , in both cases the second derivative of the function is always positive, always negative, or first positive and then negative (as it has a root $z_1$ or $z_2$).

% I.e., $F_{\mu}(q) +H\left(2q\right) 1_{q \le min\{x_0,1/4\}}$
% is either strictly convex on one part and strictly concave at the other part within , or convex/ concave on the whole interval.

I.e., $F_{\mu}(q) +H\left(2q\right)$
is either strictly convex on one part of $I_1$ (which is of the form $[0,z_1]$ and strictly concave at the other part ($[z_1,  min\{x_0,1/4\}]$), convex on the whole interval, or concave on all of $I_1.$
Similarly $F_{\mu}(q)$ is either strictly convex on one part of $I_2$, $[ min\{x_0,1/4\}, z_2]$, and strictly concave at the remaining part, $[z_2,1]$, convex on all of $I_2$, or concave on the whole interval $I_2$.

On each interval (so for both $I_1$ and $I_2$), the minimum is attained by a probability distribution that either has only one value with positive probability (if the studied function is convex), or two, one of them being the maximum of the interval.
When the latter occurs on $I_2$, one can extend $I_2$ to $[\min\{x_0,1/4\},1]$ and redistribute the mass from $1/2$ over $1$ and the inflection point $z_2$ of $I_2$ and repeat as before. Increasing the probability mass of $1$ even further decreases $\Exp'_{p \sim \mu} [H(\min(2p,1/2))]$, so the latter distribution was not a minimising probability distribution in~\cref{lem:reduced_minimizingfunction}.

This results into candidate probability distributions with at most $4$ different values with positive mass.

In total there are $3^2=9$ combinations (which one can double based on $x_0<1/4$ and $x_0 \ge 1/4$, but each such pair works by the same ideas) to consider for the behaviour on the two intervals. The 3 combinations where the considered function on $I_1$ is convex are almost immediate.
In the other situations we can modify the probability measure even further in steps and conclude at the end that a probability distribution that is a solution in~\cref{lem:reduced_minimizingfunction} has at most $3$ values with positive mass.

% the hardest combination, which is the one where the function $F_{\mu}(q) +H\left(2q\right) 1_{q \le min\{x_0,1/4\}}$ had a root on both $I_1$ and $I_2$  and $z_2$ on $I_1$ and $I_2$ for the corresponding function) and $x_0<1/4$.

% (i.e., we assume a particular form/combination and reach  conclusion if there are too many values where some mass of the probability distribution appears).

%In the latter case, due to earlier considerations about $\Pr(p=r=1)=0,$ one can even improve further.

If $min\{x_0,1/4\}=x_0$ has positive probability on the first interval and this is different from the (smallest) value $y_0$ on the second interval that got positive probability, one can repeat the argument by replacing $x_0$ by $\min\{y_0, 1/4\}.$
This implies that in case there are $4$ values with positive probability, the values $1$ and $\frac 14$ are among them.
But when we would have $x_0 \ge \frac 14,$ we know that $\Exp'[H(\min\{2p, 1/2\})$ does not depend on the distribution on the interval $[1/4,1]$ and as such, we can repeat the argument about the extremum for $F_{\mu}(q)$. 
At the end, we conclude that the support has no more than $3$ elements with positive probability.
Furthermore, if the support contains exactly $3$ elements with positive probability, at least one of them is at most $1/4.$

We illustrate this for an example where the function $F_{\mu}(q) +H\left(2q\right)$ on $I_1 \subsetneq [0,1/4]$ is both convex and concave (there is an inflection point $z_1 \in I_1$), $y_0>1/4$ and $F_{\mu}(q)$ is convex (convex and concave works similar) on $[x_0,1/2]$, in~\cref{fig:towards_better_distributions}.
Here the red dots represent the values $(q,f(q))$, where $f(q)=F_{\mu}(q) +H\left(2q\right) 1_{q \le min\{x_0,1/4\}}$, for those $q$ that have a positive probability under the candidate probability measure $\mu$.

First we extend $I_1=[0,x_0]$ to $[0,\min\{y_0, 1/4\}]$ and redistribute the mass.
We redefine $I_1$ and $I_2$ and redistribute the mass on $[1/4,1]$ (here it is within $[1/4,1/2]$), to end with a candidate probability distribution for~\cref{lem:reduced_minimizingfunction} which has less than $4$ values with positive probability.

\begin{figure}[h]
    \centering
\begin{tikzpicture}[scale=0.99]
  \draw[->] (-0.1, 0) -- (5.2, 0) node[right] {$x$};
  \draw[->] (0, -0.5) -- (0, 1) node[above] {$y$};
  \draw[domain=0:0.5, smooth, variable=\x, blue] plot ({\x}, {2*(\x-0.25)^2+0.6});
  \draw[domain=0.5:1, smooth, variable=\x, blue] plot ({\x}, {0.85 - 2*(\x-0.75)^2});
   \draw[domain=1:3, smooth, variable=\x, blue] plot ({\x}, {-1/8+1/4*(\x-2)^2});
  \draw[domain=3:5, smooth, variable=\x, blue] plot ({\x}, {3/8 - 1/4*(\x-4)^2});
  \draw[fill, color=red] (0.25,0.6) circle (0.075);
  \draw[fill, color=red] (1,0.725) circle (0.075);
  \draw[fill, color=red] (2.25,-1/8) circle (0.075);
  \draw[fill, color=red] (5,1/8) circle (0.075);
  \foreach \y in {1.25,2.5,5}{
  \draw (\y,-0.2)--(\y,0.2);
  }

   \node at (5,-0.35) {$1$};
    \node at (2.5,-0.4) {$1/2$};
     \node at (1.25,-0.4) {$1/4$};
\end{tikzpicture}
\quad
\begin{tikzpicture}[scale=0.99]
  \draw[->] (-0.1, 0) -- (5.2, 0) node[right] {$x$};
  \draw[->] (0, -0.5) -- (0, 1) node[above] {$y$};
  \draw[domain=0:0.5, smooth, variable=\x, blue] plot ({\x}, {2*(\x-0.25)^2+0.6});
  \draw[domain=0.5:1.25, smooth, variable=\x, blue] plot ({\x}, {0.85 - 2*(\x-0.75)^2});
   \draw[domain=1.25:3, smooth, variable=\x, blue] plot ({\x}, {-1/8+1/4*(\x-2)^2});
  \draw[domain=3:5, smooth, variable=\x, blue] plot ({\x}, {3/8 - 1/4*(\x-4)^2});
  \draw[fill, color=red] (0.5,0.725) circle (0.075);
  \draw[fill, color=red] (1.25,0.35) circle (0.075);
  \draw[fill, color=red] (2.25,-1/8) circle (0.075);
  \draw[fill, color=red] (5,1/8) circle (0.075);
  \foreach \y in {1.25,2.5,5}{
  \draw (\y,-0.2)--(\y,0.2);
  }
     \node at (5,-0.35) {$1$};
    \node at (2.5,-0.4) {$1/2$};
     \node at (1.25,-0.4) {$1/4$};
\end{tikzpicture}
\quad
\begin{tikzpicture}[scale=0.99]
  \draw[->] (-0.1, 0) -- (5.2, 0) node[right] {$x$};
  \draw[->] (0, -0.5) -- (0, 1) node[above] {$y$};
  \draw[domain=0:0.5, smooth, variable=\x, blue] plot ({\x}, {2*(\x-0.25)^2+0.6});
  \draw[domain=0.5:1.25, smooth, variable=\x, blue] plot ({\x}, {0.85 - 2*(\x-0.75)^2});
   \draw[domain=1.25:3, smooth, variable=\x, blue] plot ({\x}, {-1/8+1/4*(\x-2)^2});
  \draw[domain=3:5, smooth, variable=\x, blue] plot ({\x}, {3/8 - 1/4*(\x-4)^2});
  \draw[fill, color=red] (0.5,0.725) circle (0.075);
  \draw[fill, color=red] (1.25,0.015625) circle (0.075);
  \draw[fill, color=red] (2.25,-1/8) circle (0.075);
  \draw[fill, color=red] (5,1/8) circle (0.075);
  \foreach \y in {1.25,2.5,5}{
  \draw (\y,-0.2)--(\y,0.2);
  }
     \node at (5,-0.35) {$1$};
    \node at (2.5,-0.4) {$1/2$};
     \node at (1.25,-0.4) {$1/4$};
\end{tikzpicture}
\quad
\begin{tikzpicture}[scale=0.99]
  \draw[->] (-0.1, 0) -- (5.2, 0) node[right] {$x$};
  \draw[->] (0, -0.5) -- (0, 1) node[above] {$y$};
  \draw[domain=0:0.5, smooth, variable=\x, blue] plot ({\x}, {2*(\x-0.25)^2+0.6});
  \draw[domain=0.5:1.25, smooth, variable=\x, blue] plot ({\x}, {0.85 - 2*(\x-0.75)^2});
   \draw[domain=1.25:3, smooth, variable=\x, blue] plot ({\x}, {-1/8+1/4*(\x-2)^2});
  \draw[domain=3:5, smooth, variable=\x, blue] plot ({\x}, {3/8 - 1/4*(\x-4)^2});
  \draw[fill, color=red] (0.5,0.725) circle (0.075);

  \draw[fill, color=red] (1.75,-1/8) circle (0.075);
  \draw[fill, color=red] (5,1/8) circle (0.075);
  \foreach \y in {1.25,2.5,5}{
  \draw (\y,-0.2)--(\y,0.2);
  }
     \node at (5,-0.35) {$1$};
    \node at (2.5,-0.4) {$1/2$};
     \node at (1.25,-0.4) {$1/4$};
\end{tikzpicture}
    \caption{An example of improving the distribution}
    \label{fig:towards_better_distributions}
\end{figure}

\subsection{Verification for distributions with support of size at most $3$}\label{subsec:ver_support3}

Once the support is reduced to $3$ elements, $\{a_1,a_2,1\}$, by knowing the associated probabilities $p_1, p_2, 1-p_1-p_2$ of each element, the inequality in Question~\ref{ques:Sawin} can be checked.
As such, we find an optimisation problem in $4$ variables.
Using Maple, it has been checked in multiple ways; by solving a minimisation problem in multiple regimes, and by plotting an implicit plot, as well as plots with a fixed choice for $a_1$ assuming $\Exp[p]=c$.
From these, we note that there are two local regions where the minima occur; around $a_1=0$ and around $p_1=0$\footnote{See \url{https://github.com/StijnCambie/UCconjecture}, documents \text{FinalComputation23}}, corresponding with the cases where $p$ is $\{0,1\}$-valued and the atomic one used to show sharpness in Subsection~\ref{subsec:uppbound}.

Finally, we also give a more rigorous proof for the case where the distribution is atomic, i.e., $\Pr(p=b)=1-a$ and $\Pr(p=1)=a$ and $\Exp[p]=a+(1-a)b \le c$, where $c\sim 0.3823455$ is the claimed optimum.
Then $\Exp[H(p+q-pq)]=(1-a)^2 h(2b-b^2), \Exp[H(p)]=(1-a)h(b)$ and $\Exp\left[H\left(max\left(p,r,\min\left(p+r,1/2\right)\right)\right)\right] \ge (1-2a)h( \min(2b,1/2) ).$
Since the case where $\Pr(p=r=1)=0$ is the worst case, we need to show that 
$$(1-\a)(1-a)^2 h(2b-b^2) + \a(1-2a)h( \min(2b,1/2) )-(1-a)h(b) \ge 0,\mbox{ or equivalently}$$
$$
(1-\a) h(2b-b^2)  a^2 + \left( -2(1-\a) h(2b-b^2) -2\a h( \min(2b,1/2) )+h(b) \right)a +O_{b, \alpha}(1) \ge 0 $$
For fixed (non-zero) $b$, this is a quadratic function in $a$ with positive leading coefficient, which attains its minimum at 
$a= 1+\frac{ 2\a h( \min(2b,1/2) )-h(b)}{2(1-\a)h(2b-b^2)}>\frac{c-b}{1-b}$ and thus it is sufficient to prove this in the case where $a=\frac{c-b}{1-b}.$\footnote{See \url{https://github.com/StijnCambie/UCconjecture/blob/main/Sharpness.sagews}}

\subsection{Precise verification}\label{subsec:ver_combinedideas}

If we combine our conclusions from Subsections~\ref{subsec:redSup_to4} and~\ref{subsec:reduction} with the one from~\cite{Yu22}, we obtain that there are two possible forms for the joint distribution of $(p,r).$
Either $p=r$ and the support of $p$ has size bounded by $2$ (the elements being bounded by $1/2$), or the support has $3$ elements $\{a_1,a_2,1\}$ and $p_1=1-2p_2$, where $\Pr(p=a_2, r=1)=\Pr(p=1, r=a_2)=p_2$.
Equivalently, with the notation of~\cite{Yu22}, the distribution $P_{pr}$ is of the form $(1-\beta) Q_{a,a}+ \beta Q_{b,b}$ or $(1-\beta) Q_{a,a}+ \beta Q_{1,b}$ (where $a\le b$).

Hereby for a fixed choice of $c$, $\beta$ is a function of $a$ and $b.$
In the first case, $a<c<b$ and $\beta=\frac{c-a}{b-a}$. In the second case, $\beta=\frac{2(c-a)}{1+b-2a}.$

As such, the final verification for~\cref{ques:Sawin} can be deduced from an inequality involving only two variables.
This final verification has been done in~\url{https://github.com/StijnCambie/UCconjecture}, documents \text{FinalComputation24}\footnote{For some reason, minus is replaced by $K$ in the PDF.},
If $p=r$, the inequality is strict.
In the case with the support containing $1$, we deduce that the atomic distribution from~\cref{subsec:uppbound} is the (unique) minimiser, and the inequality is true and tight.

\section{Proof of bound for sharper union-closed conjecture}\label{sec:proof}

Having established the answer to~\cref{ques:Sawin} in the previous section, we now present the formal proof of Theorem~\ref{thr:improvement}, as sketched in~\cite{Sawin22}, to complete the exposition
Let $\a \sim 0.0356069$ and  $c\sim 0.3823455$ be the previously determined optimal constants for Question~\ref{ques:Sawin}.

\begin{proof}[Proof of Theorem~\ref{thr:improvement}]
Assume there is a (nonempty) union-closed family $\F \subset 2^{[n]}$ for which every element $i \in [n]$ appears in at most a $c$-fraction of the sets in $\F$. Without loss of generality, we can assume that $1$ appears in at least one set in $\F.$
    We consider random variables $A,B,C$, which are three uniform samples from $\F$, defined as follows.
    The uniform sampling of $B$ happens independently of the sampling of $A$ and $C.$ 
    The latter two are sampled element-wise and in a dependent way.
    We denote $A_i=1$ if $i\in A$ and otherwise $A_i=0$, i.e., it is the indicator function $1_{i \in A}$,
    and $A_{<i}=(A_1, \ldots, A_{i-1})$ is the sequence of the first $i-1$ indicator random variables.
    Analogously $C_i$ and $C_{<i}$ are defined.
    
  For every $i \in [n]$ and given (fixed) realisations 
  $a_{<i}=(a_1, \ldots, a_{i-1})$ and $c_{<i}=(c_1, \ldots, c_{i-1})$, we consider the fractions
    $$f_a = \frac{\abs{\{ S \in \F \colon i \in S, S_{<i}=a_{<i} \} } }{\abs{\{ S \in \F \colon S_{<i}=a_{<i} \} } } \mbox{ and } f_c = \frac{\abs{\{ S \in \F \colon i \in S, S_{<i}=c_{<i} \} } }{\abs{\{ S \in \F \colon S_{<i}=c_{<i} \} } }.$$
    If $\max\{f_a, f_c\}>0.5,$
    we take $x \in U([0,1])$, a uniformly random element from $[0,1]$, and take $a_i=[x<f_a]$ and $c_i=[x<f_c]$.
    That is, $a_i=1$ if $x<f_a$ and otherwise $a_i=0.$ 
    Similarly, if $f_a, f_c\le 1/2,$ we take $a_i=[x<f_a]$ and $c_i=[0.5-f_c <x<0.5]$ for the uniform random generated $x \in [0,1].$ 

    If we do the previous steps for every $i \in [n],$ there are up to $2^{i-1}$ different realisations and fractions. 
    
    Let $p_i=\Pr(A_{<i+1} \mid A_{<i})$ be the conditional probability distribution, associated with the probability (fraction) for $a_{<i+1}$ given any realisation of $a_{<i}.$
    Define $r_i=\Pr(C_{<i+1} \mid C_{<i})$ completely analogous.
    
    Then with the above steps for concrete realisations, for the random variable $A\cup C,$
    we have $$\Pr[ (A\cup C)_i \mid A_{<i}, C_{<i} ]=\max\{ p_i, r_i, \min(p_i+r_i,1/2)\}.$$
    By the product rule applied to the conditional probabilities $p_i$, we have that $A$ (and similarly $C$) will be uniformly distributed over $\F$, that is for a particular set $T \in \F,$ we have 
    $$\Pr(A=T) = \prod_{i \in T} \frac{\abs{\{ S \in \F \colon i \in S, S_{<i}=T_{<i} \} } }{\abs{\{ S \in \F \colon S_{<i}=T_{<i} \} } }
    \cdot \prod_{i \not\in T} \frac{\abs{\{ S \in \F \colon i \not \in S, S_{<i}=T_{<i} \} } }{\abs{\{ S \in \F \colon S_{<i}=T_{<i} \} } }=\frac{1}{\abs\F}.$$
    Hence $A,B,C$ all have the (same) uniform probability distribution.
    Let $q_i= \Pr(i\in B \mid B_{<i})$ be a conditional probability distribution.
    Now $p_i, q_i$ and $r_i$ are identically distributed conditional probability distributions, where $q_i$ is independent from the other two, but $p_i$ and $q_i$ are dependent.
    
    The chain rule and data processing inequality respectively yield
    $$H( (A \cup B)_{<i+1} )=H( (A \cup B)_{<i} )+H( (A \cup B)_{<i+1} \mid (A \cup B)_{<i} ) \ge H( (A \cup B)_{<i} )+H( (A \cup B)_{<i+1} \mid A_{<i},B_{<i} ),$$
    while $H(A_{<i+1})=H(A_{<i})+H(A_{<i+1}\mid A_{<i}).$
    As such to prove that 
    $$(1-\a) H(A \cup B)+\a H(A\cup C) > H(A),$$
    it is sufficient to prove that 
     $(1-\a)H( (A \cup B)_{1} )+\a H( (A \cup C)_{1}) > H(A_{1}) $ and 
    $(1-\a)H( (A \cup B)_{<i+1} \mid A_{<i},B_{<i} )+\a H( (A \cup C)_{<i+1} \mid A_{<i},C_{<i} ) \ge H(A_{<i+1}\mid A_{<i}) $ for every $i\ge 2.$
    %But now $H( (A \cup B)_{<i+1} \mid A_{<i},B_{<i} )=H( (A \cup B)_{i} \mid A_{<i},B_{<i} )$
    But with the conditional probability distribution of $(A\cup B)_i$, with which we refer to $\Pr((A\cup B)_{<i+1} \mid (A\cup B)_{<i} )$, being $p_i+q_i-p_iq_i$ (by the principle of inclusion-exclusion) and of $(A\cup C)_i$ being $\max\{ p_i,r_i,\min(p_i+r_i,0.5)\}$ (by the choice of the samples), where $p_i,q_i,r_i$ all have expectation less than $c$, this follows from the answer to Question~\ref{ques:Sawin}. Since equality cannot appear for $i=1$, the inequality is strict. 
    We conclude that $\max\{H(A \cup C), H(A \cup B)\} > H(A)=\log_2 \abs \F$, which is a contradiction. So no such family $\F$ as initially assumed exists.
\end{proof}

\subsection*{Acknowledgement}

We thank an anonymous referee for their careful reading and valuable suggestions, including critical remarks on readability that helped improve the presentation of the paper. Their recommendation to consider connections with the work of Yu~\cite{Yu22} led to the addition of Subsection~\ref{subsec:ver_combinedideas}, strengthening the resolution of~\cref{ques:Sawin}.

\paragraph{Open access statement.} For the purpose of open access,
a CC BY public copyright license is applied
to any Author Accepted Manuscript (AAM)
arising from this submission.

\end{document}